\documentclass[11pt]{amsart}

\usepackage{amsmath,amssymb}
\usepackage{graphicx}
\usepackage{color}

\newlength{\dhatheight}

\newtheorem{theorem}{Theorem}[section]

\newtheorem{proposition}[theorem]{Proposition}
\newtheorem{lemma}[theorem]{Lemma}

\newtheorem{corollary}[theorem]{Corollary}
\newtheorem{definition}[theorem]{Definition}

\newcommand{\dd}{\Delta}
\newcommand{\ra}{\rightarrow}

\newcommand{\ga}{\Gamma}

\newcommand{\pp}{{\mathcal{P}}}

\newcommand{\nf}{{\mathcal{N}}}

\newcommand{\as}{{autostackable}}

\newcommand{\vkd}{van Kampen diagram}

\newcommand\ves{{\vec E_{r}}}  
\newcommand\dgd{{\vec E_d}}  

\newcommand{\ega}{e_{g,a}}

\newcommand{\stf}{\phi}
\newcommand{\ff}{\Phi}

\newcommand{\wstf}{stacking function}
\newcommand{\wff}{flow function}

\newcommand{\wisf}{induced stacking function}
\newcommand{\wiff}{induced flow function}

\newcommand{\sr}{synchronously regular}

\newcommand{\asa}{asynchronously automatic}
\newcommand{\aau}{asynchronous automaton}
\newcommand{\are}{asynchronously regular}

\newcommand{\prs}{prefix-rewriting system}
\newcommand{\cpr}{convergent prefix-rewriting system}
\newcommand{\bpr}{bounded prefix-rewriting system}
\newcommand{\bcpr}{bounded convergent prefix-rewriting system}
\newcommand{\pro}{processed}
\newcommand{\fcr}{finite convergent rewriting system}

\newcommand{\sym}{inverse-closed}
\newcommand{\pad}{\mu}

\newcommand{\ep}{\lambda}
\newcommand{\gi}{1}
\newcommand{\ew}{\lambda}

\newcommand{\qgood}{Q_a^{good}}

\begin{document}
\title[Algorithms and topology of Cayley graphs for groups]
{Algorithms and topology of Cayley graphs for groups}

\author[M.~Brittenham]{Mark Brittenham}
\address{Department of Mathematics\\
        University of Nebraska\\
         Lincoln NE 68588-0130, USA}
\email{mbrittenham2@math.unl.edu}

\author[S.~Hermiller]{Susan Hermiller}
\address{Department of Mathematics\\
        University of Nebraska\\
         Lincoln NE 68588-0130, USA}
\email{smh@math.unl.edu}

\author[D.~Holt]{Derek Holt}
\address{Mathematics Institute\\
        University of Warwick\\
         Coventry CV4 7AL, UK}
\email{D.F.Holt@warwick.ac.uk}

\begin{abstract}
Autostackability for
finitely generated groups
is defined via a topological
property of the associated Cayley graph 
which can be encoded in a finite
state automaton.  
Autostackable groups have solvable word problem
and 
an effective inductive procedure for constructing
van Kampen diagrams with respect to a
canonical finite presentation.  
A comparison with 
automatic groups is given.  
Another characterization
of autostackability is given in terms of {\prs}s.
Every group which admits a finite complete rewriting
system or an \asa\ structure with
respect to a prefix-closed set of normal forms
is also \as.  As a consequence, the fundamental
group of every closed 3-manifold with any
of the eight possible uniform geometries is \as.
\end{abstract}

\maketitle


\section{Introduction}\label{sec:intro}


A primary motivation for the definition of the class of
automatic groups is to
make computing the word problem 
for 3-manifold groups tractable; however, in their
introduction of the theory of automatic groups,
Epstein, {\em et.~al.}~\cite{echlpt}
showed that the fundamental
group of a closed 3-manifold having Nil or Sol geometry is
not automatic.  Brady~\cite{brady} showed that there are
Sol geometry groups that do not belong to the
wider class of \asa\ groups.  Bridson and Gilman~\cite{bridgil}
further relaxed the language theoretic restriction
on the associated normal forms, replacing 
regular with indexed languages, and showed 
that every 3-manifold has an asynchronous
combing with respect to an indexed language.
More recently, Kharlampovich, Khoussainov, 
and Miasnikov~\cite{kkm} have defined the class of
Cayley automatic groups, extending the notion of an automatic structure
(preserving the regular language restriction), 
but it is as yet unknown whether all Nil and Sol 3-manifold
groups are Cayley automatic.  In this paper we
define the notion of autostackability for finitely generated
groups using properties very closely related to automatic
structures, that holds for 3-manifold groups of all
uniform geometries.

Let $G$ be a group with an \sym\ finite generating set $A$, and
let $\ga=\ga(G,A)$ be the associated 
Cayley graph.  Let $\vec E$ be the set of directed edges;
for each $g \in G$ and $a \in A$, let $\ega$ denote
the directed edge of $\ga$ with initial vertex $g$,
terminal vertex $ga$, and label $a$.
Let $\nf \subset A^*$ be a set of 
normal forms for $G$ over $A$;
for each $g \in G$, we denote the normal
form word representing $g$ by $y_g$.
Note that whenever we have an equality of
words $y_ga=y_{ga}$ or $y_g=y_{ga}a^{-1}$,
then there is a \vkd\ for the word 
$y_gay_{ga}^{-1}$ 
that contains no 2-cells; in this case we call
the edge $\ega$
{\em degenerate}.  
Let $\vec E_{\nf,d} = \dgd$
be the set of all degenerate directed edges,
and let $\vec E_{\nf,r} = \ves := \vec E \setminus \dgd$; we 
refer to elements of $\ves$ as {\em recursive} edges.

\begin{definition}\label{def:as}
A group $G$ with finite \sym\ generating
set $A$ is {\em autostackable}
if there are a set $\nf$ of normal forms 
  for $G$ over $A$ 
containing the empty word, 
a constant $k$, and
 a function $\stf:\nf \times A \rightarrow A^*$
such that the following hold:
\begin{enumerate} 
\item 
The graph of the function $\stf$,
$$
graph(\stf):=\{(y_g,a,\stf(y_g,a)) \mid g \in G, a \in A\},
$$ 
is a \sr\ language.
\item
For each $g \in G$ and $a \in A$, the word $\stf(y_g,a)$ 
has
length at most $k$
and represents
the element $a$ of $G$, 
and:
\begin{itemize}
\item[(2d)] If $\ega \in \vec E_{\nf,d}$, then 
the equality of words $\stf(y_g,a) = a$ holds.
\item[(2r)] 
The transitive closure $<_\stf$ of the relation $<$ on
$\vec E_{\nf,r}$, defined by 
\begin{itemize}
\item[]
$e' < \ega$ whenever $\ega,e' \in \vec E_{\nf,r}$ and $e'$ is 
on the directed path in $\ga$ 
labeled $\stf(y_g,a)$ starting at the 
vertex $g$
\end{itemize}
is a strict well-founded 
partial ordering.
\end{itemize}
\end{enumerate}
\end{definition}

Removing the algorithmic property in (1),
the group $G$ is called {\em stackable} over the \sym\ 
generating set $A$ if 
property (2) holds for some
normal form set $\nf$
(containing $\ew$),
constant $k$, and function $\stf:\nf \times A \ra A^*$. 
In~\cite{bh}, the first two authors define and study 
the class of stackable groups.  
In~\cite[Lemma~1.5]{bh} they show that stackability
implies that the finite set $R_c$ of words of
the form $\stf(y_g,a)a^{-1}$ (for $g \in G$ and $a \in A$)
is a set of defining relators for $G$, and
the set $\nf$ of normal forms is closed
under taking prefixes.  
Hence the set $\nf$ uniquely determines a
maximal tree in the Cayley graph $\ga$,
consisting of the edges that lie on paths
labeled by words in $\nf$.

This leads to a topological description 
of the concept of autostackability.
Let $T$ be a maximal tree in $\ga$.  
For each $g \in G$ and $a \in A$,
we view the two directed edges $e_{g,a}$ and $e_{ga,a^{-1}}$  of $\ga$ to
have a single underlying undirected edge in $\ga$.
Let $\vec P$ be the set of all finite
length directed edge paths in $\ga$.
A {\em flow} function
associated to $T$ is a
function $\ff:\vec E \ra \vec P$ 
satisfying the properties that: 
\begin{itemize}
\item[(a)] For each edge $e \in \vec E$,
the path $\ff(e)$ has the same initial and terminal
vertices as $e$.
\item[(b-d)] If the undirected edge underlying $e$ 
lies in the tree $T$, then $\ff(e)=e$.
\item[(b-r)]  The transitive closure 
$<_\ff$ of the relation $<$ on
$\vec E$, defined by 
\begin{itemize}
\item[]
$e' < e$ whenever $e'$ lies on the path $\ff(e)$
and the undirected edges underlying both
$e$ and $e'$ do not lie in $T$,
\end{itemize}
is a strict well-founded 
partial ordering.
\end{itemize}
That is, the map $\ff$ fixes the edges lying in the tree $T$
and describes a ``flow'' of the
non-tree edges toward the tree (or toward the basepoint).
A flow function is {\em bounded} if there is
a constant $k$ such that for all $e \in \vec E$,
the path $\ff(e)$ has length at most $k$.

For each element $g \in G$, let $y_g$ be the unique
word labeling a geodesic path in the tree $T$ from
the identity element $\gi$ of $G$ to $g$, and let
$\nf_T := \{y_g \mid g \in G\}$ be the corresponding set
of normal forms.  
Let $\beta_T:\nf_T \times A \ra \vec E$
denote the natural bijection defined by 
$\beta_T(y_g,a):=\ega$, and let $\rho:\vec P \ra A^*$
be the function that maps each directed path to
the word labeling that path in $\ga$.  
The 
composition $\rho \circ \ff \circ \beta_T:\nf \times A \ra A^*$
is part of a stackable structure for $G$ over $A$,
which we call the {\em \wisf}.
Conversely,~\cite[Lemma~1.5]{bh} implies that given a
{\wstf} $\stf:\nf \times A \ra A^*$ from a stackable structure,
there is an {\em \wiff} $\ff:\vec E \ra \vec P$, such that
$\ff(\ega)$ is the path in $\ga$ starting at the vertex $g$
labeled by the word $\stf(y_g,a)$.
Thus we have the following characterizations.

\begin{proposition}\label{prop:asflow}
Let $G$ be a group with a finite \sym\ generating set $A$.
(1) The group $G$
is stackable over $A$ if and only if the Cayley graph
$\ga(G,A)$ admits a maximal tree with an associated bounded flow function. \\
(2) The group $G$ is 
autostackable over $A$ if and only if there exists a
maximal tree in $\ga(G,A)$
with a bounded flow function such that
the graph of the \wisf\ 
is \sr.
\end{proposition}


In Section~\ref{sec:bkgd} of this paper, we give 
definitions and notation, and discuss background on
normal forms, van Kampen diagrams, and 
language theory.

Section~\ref{sec:aut} contains a comparison
of the definitions 
for autostackable groups versus
automatic groups.
We contrast word problem solutions and
van Kampen diagram constructions for these
two classes of groups.
In analogy with the relationship between autostackable
and stackable groups above, removing the algorithmic
Property (i) of Definition~\ref{def:am} of automaticity
yields the definition of combable groups.
We show how to modify the proof 
of~\cite[Propositions~1.7,1.12]{bh}
to show the following.

\smallskip

\noindent{\bf Proposition~\ref{prop:solvwp}.}
{\em Autostackable groups are finitely presented,
have solvable word problem, and admit
a recursive algorithm to build a van Kampen diagram
for each word representing the identity element.}

\smallskip

The class of automatic groups is strictly contained
in the class of \asa\ groups; in
Section~\ref{sec:asa}, 
we consider 
this larger class.

\smallskip

\noindent{\bf Theorem~\ref{thm:asa}.} 
{\em Every group that has an \asa\ 
structure with a prefix-closed normal form set
is autostackable.  
}

\smallskip

We note that although 
Epstein et.~al.~\cite[Theorems~2.5.1,5.5.9]{echlpt}
have shown that
every automatic group has an automatic structure
with respect to a set of normal forms, and also
an automatic structure with respect to a
prefix-closed set of not necessarily unique
representatives, it is an open 
problem~\cite[Open~Question~2.5.20]{echlpt} whether there
must be an automatic structure on a prefix-closed
set of normal forms. 
Gilman has given other
characterizations of groups that are 
automatic with respect to a prefix-closed normal form
set in~\cite{gilman}.
Groups known to have an automatic structure 
with respect to prefix-closed normal forms include 
finite groups~\cite{echlpt}, 
virtually abelian (and hence Euclidean)
groups and word
hyperbolic groups~\cite{echlpt},
Coxeter groups~\cite{brinkhow}, 
Artin groups of finite type~\cite{charney} and of
large type~\cite{peifer},\cite{holtrees}, and
small cancellation groups satisfying conditions
$C''(p)-T(q)$ for $(p,q) \in \{(3,6),(4,4),(6,3)\}$~\cite{johnsgard}.
The class of automatic groups with respect to prefix-closed
normal forms is closed under graph 
products~\cite[Theorem~B]{hmeier}
and finite extensions~\cite[Theorem~4.1.4]{echlpt}.

\smallskip

In Section~\ref{sec:fcrs}, we give a purely
algorithmic characterization of autostackability,
using another
type of word problem solution, namely
`prefix-sensitive rewriting'.
A {\em \cpr} for a group $G$ consists of a finite
set $A$ together with a 
subset $R \subset A^* \times A^*$ such that as
a monoid, $G$ is presented by
$G=Mon\langle A \mid u=v \text{ whenever } (u,v) \in R\rangle$,
and the rewriting operations of the form
$uz \ra vz$ for all $(u,v) \in R$ and $z \in A^*$
satisfy:
\begin{itemize}
\item {\em Normal forms:} Each $g \in G$ is 
represented by exactly one {\em irreducible} word 
(i.e. word that cannot be rewritten)
over $A$.
\item {\em Termination:} There does not exist an infinite
sequence of rewritings $x \ra x_1 \ra x_2 \ra \cdots$.
\end{itemize}
A \prs\ is {\em bounded} if there exists a
constant $k$
such that for each pair $(u,v) \in R$, there  
are words $s,t,w \in A^*$ with $s$ and $t$ of length
at most $k$ such that
$u=ws$ and $v=wt$. 

\smallskip

\noindent{\bf Theorem~\ref{thm:prs}.}
{\em Let $G$ be a finitely generated group. \\
(1) The group $G$ 
is stackable if and only if $G$ admits a bounded
\cpr.\\
(2) The group $G$ 
is autostackable if and only if $G$ admits a \sr\ bounded
\cpr.
}

\smallskip 

As part of the proof of Theorem~\ref{thm:prs},
in Proposition~\ref{prop:min}, we show that given
any \sr\ \bcpr\ $R$ for $G$, there is a subset $Q'$ of $R$
that is a \sr\ \bpr\ for $G$
such that for every $(u,v) \in Q'$, every proper
prefix of $u$ is irreducible over $R$, and no
two distinct word pairs in $Q'$ have the same left hand side.

In contrast to these results, Otto~\cite[Corollary~5.3]{otto}
has shown that a group is automatic with respect to
a prefix-closed set of normal forms over a monoid
generating set $A$ if and only if there
exists a \sr\ \cpr\ such that for every $(u,v) \in R$,
the word $v$ is irreducible over $R$, and the word
$u$ is irreducible over all of the other rewriting rules of $R$.

Synchronously regular bounded {\cpr}s are a generalization
of the more widely studied concept of
finite convergent (also called complete) rewriting systems, 
which admit rewriting operations of the form
$wuz \ra wvz$ whenever
$(u,v) \in R$ and $w,z \in A^*$.
Thus Theorem~\ref{thm:prs} yields:

\smallskip

\noindent{\bf Corollary~\ref{cor:fcrs}.} {\em Every
group that admits a \fcr\ 
is autostackable.}

\smallskip

Groups known to have a \fcr\ 
include
finite groups, 
alternating knot groups~\cite{chour},
surface groups~\cite{lechen}, 
virtually abelian groups, polycyclic
groups, 
and more generally
constructible solvable groups~\cite{grovessmith}, 
Coxeter groups of large type~\cite{hermiller},
and Artin groups of finite type~\cite{hmeierartin}
(see also Le Chenadec's~\cite{lechenbook} text
for many more examples).
This class of groups is closed under
graph products~\cite{hmeier}, 
extensions~\cite{grovessmith},\cite{hmeierartin},
and certain amalgamated products and
HNN extensions~\cite{grovessmith}.

The iterated Baumslag-Solitar
groups presented by
$\langle a_0,a_1,...,a_k \mid a_0^{a_1}=a_0^2,...,
a_{k-1}^{a_k}=a_{k-1}^2\rangle$ were shown by
Gersten~\cite[Section~6]{gerstenexpid} to have
Dehn function asymptotic to a
k-fold iterated exponential function, and
also to have a \fcr\ 
(see~\cite{hmeiermeastame} for details).  
The following is then an immediate consequence of
the results above.

\begin{corollary}\label{cor:iteratedexp}
The class of autostackable groups
includes groups whose Dehn functions'
growth is asymptotically equivalent to
an iterated exponential function with
arbitrarily many iteration steps.
\end{corollary}

\noindent  This result
is in strong contrast to the quadratic upper
bound on the Dehn function for any
automatic group~\cite[Theorem~2.3.12]{echlpt}.

Miller~\cite[p.~31]{miller} has shown that there
exists a split extension of a finitely generated
free group by another finitely generated free group
that has unsolvable conjugacy problem.  Since free
groups admit finite complete rewriting systems,
the results above also give the following.

\begin{corollary}\label{cor:conjpbm}
The class of autostackable groups
includes groups with unsolvable conjugacy problem.
\end{corollary}

Finally, we return to the motivation of computing
the word problem in 3-manifold groups.
In~\cite{hs}, Hermiller and Shapiro show that
if $M$ is a closed 3-manifold with uniform
geometry that is not hyperbolic, then $\pi_1(M)$
has a \fcr.  On the
other hand, Epstein et.~al.~\cite{echlpt} show that
every word hyperbolic group, and hence every
hyperbolic 3-manifold fundamental group,
is automatic with respect to a shortlex,
and hence prefix-closed, set of normal forms.  
Hence we obtain the following.

\begin{corollary}\label{cor:3mfd}
Every fundamental group of a closed 3-manifold
with uniform geometry is autostackable.
\end{corollary}


\section{Notation and background} \label{sec:bkgd}


Throughout this paper, 
let $G$ be a group with a finite generating set $A$ that
is closed under inversion,
and let $\ga$ be 
the associated Cayley graph.
Let $A^*$ denote the free monoid, 
i.e.~the set of all finite words over $A$,
and let $\pi:A^* \ra G$ denote the canonical
surjection.
Whenever $u$ and $v$ lie
in the set $A^*$ of all words over $A$, 
we write $u=v$ if $u$ and $v$ are the same word,
and $u=_Gv$ if $u$ and $v$ represent the same
element of $G$; i.e., if $\pi(u)=\pi(v)$.  
Let $\gi$ denote the identity
element of $G$ and let $\ew$ denote the empty word in $A^*$;
then $\pi(\ew)=\gi$.

Given a word $w \in A^*$, let $l(w)$ denote the
length of $w$ as a word over $A$.
For each $a \in A$, the symbol $a^{-1}$
represents another element of $A$,
and so for each word $u=a_1 \cdots a_m$ in $A^*$
with each $a_i$ in $A$, there
is a formal inverse word $u^{-1} := a_m^{-1} \cdots a_1^{-1}$
in $A^*$.



\subsection{Normal forms and van Kampen diagrams} \label{sub:nfvkd}


A set $\nf$ of {\em normal forms} for $G$ over $A$
is a subset of the set $A^*$
such that the restriction of the canonical 
surjection $\pi:A^* \ra G$ 
to $\nf$ is a bijection. 
As in Section~\ref{sec:intro},
the symbol $y_g$ denotes the normal form for $g \in G$;
by slight abuse of notation, we use the symbol
$y_w$ to denote the normal form for $\pi(w)$
whenever $w \in A^*$.

Given a set $R$ of defining relators for a group $G$,
so that $\pp = \langle A \mid R \rangle$ is a presentation
for $G$, then
for an arbitrary word $w$ in $A^*$
that represents the
identity element $\gi$ of $G$, there is a {\em van Kampen
diagram} (or Dehn diagram) $\dd$ for $w$ with respect to $\pp$.  
That is, $\dd$ is a finite,
planar, contractible combinatorial 2-complex with 
edges directed and
labeled by elements of $A$, satisfying the
properties that the boundary of 
$\dd$ is an edge path labeled by the
word $w$ starting at a basepoint 
vertex $*$ and
reading counterclockwise, and every 2-cell in $\dd$
has boundary labeled (in some orientation)
by an element of $R$.  
See~\cite{bridson}
or~\cite{lyndonschupp} for more details on
the theory of van Kampen diagrams.

Let $\nf$ be a set of normal forms for 
$G$ over $A$
such that 
each word $w \in \nf$ labels a simple path.
For example, this property holds if $\nf$ is closed
under taking prefixes of words.
The ``seashell'' method to construct 
a van Kampen diagram
(with respect to the presentation $G=\langle A \mid R\rangle$)
for any word $w=b_1 \cdots b_n \in A^*$ that represents 
the identity of $G$ is as follows.
For each $i$ we denote the normal form word
$y_i:=y_{b_1 \cdots b_i}$.
Let $\Delta_i$ be a van Kampen diagram for the
word 
$y_{i-1}b_iy_i^{-1}$.
By successively gluing these diagrams along
the simple normal form paths along their boundaries,
we obtain a planar van Kampen diagram for $w$;
see Figure~\ref{fig:seashell} for an idealized picture.
\begin{figure}
\begin{center}
\includegraphics[width=3.6in,height=1.6in]{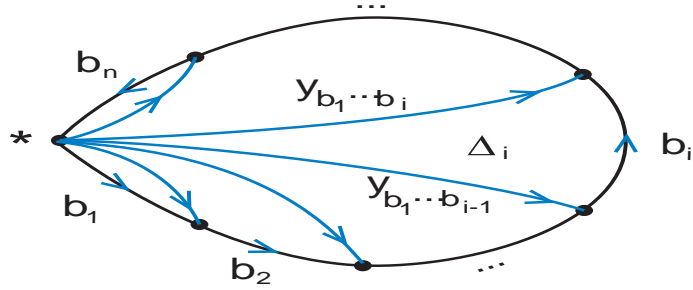}
\caption{Van Kampen diagram built with seashell method}\label{fig:seashell}
\end{center}
\end{figure}  
(See for example~\cite{echlpt},~\cite{riley}, 
or~\cite{bh} for more details.)


\subsection{Regular languages} \label{sub:reg}


For more details and proofs of the material
in this subsection, we refer the
reader to~\cite{echlpt} or~\cite{hu}.

A {\em language} over a finite set $A$ is 
a subset of the set $A^*$ of all finite words over $A$.
We also refer to subsets
of $(A^*)^n$ as languages over $A^n$.
The set $A^+$ denotes the language $A^* \setminus \{\ew\}$
of all nonempty words over $A$,
and $A^{\le k}$ denotes the finite language 
of all words over $A$ of length at most $k$.

The {\em regular} languages over $A$ are the
subsets of $A^*$ obtained from the finite subsets
of $A^*$ using finitely many operations from among
union, intersection, complement, concatenation
($S \cdot T := \{vw \mid v \in S$ and $w \in T\}$),
and Kleene star ($S^0:=\{\ew\}$, $S^n := S^{n-1} \cdot S$ and
$S^* := \cup_{n=0}^\infty S^n$).  
A {\em finite state automaton}, or FSA, is a 5-tuple
$M:=(A,Q,q_0,P,\delta)$, where
$Q$ is a finite set called the
set of {\em states}, 
$q_0 \in Q$ is the {\em initial state},
$P \subseteq Q$ is the set of {\em accept states},
and $\delta:Q \times A \ra Q$ is the
{\em transition function}.
The map $\delta$ extends to a function (often
given the same label)
$\delta:Q \times A^* \ra Q$ by recursively
defining $\delta(q,wx):=\delta(\delta(q,w),x)$
whenever $q \in Q$, $w \in A^*$, and $x \in A$.
A word $w \in A^*$
is in the language accepted by $M$ if and only if 
the state 
$\delta(q_0,w)$
lies in the set $P$.
A language $L$ over $A$ is regular if and only if 
$L$ is the language accepted by a finite state
automaton.

The class of regular
languages is closed under both image and
preimage via monoid homomorphisms
(see, for example,~\cite[Theorem~3.5]{hu}).
The class of regular sets is also closed under quotients
(see~\cite[Theorem~3.6]{hu}); we write out a special case of 
this in the following lemma for use in later sections of this
paper.

\begin{lemma}\label{lem:peel}~\cite[Theorem~3.6]{hu})
If $A$ is a finite set, $L \subseteq A^*$ is a regular
language, and $w \in A^*$,
then the quotient language
$L/w := \{x \in A^* \mid xw \in L\}$ is also a 
regular language.
\end{lemma}

Let $\$$ be a symbol not contained in $A$.
The set $A_n:=(A \cup \{\$\})^n \setminus \{(\$,...,\$)\}$
is the {\em padded $n$-tuple alphabet} derived from $A$.
For any $n$-tuple of words $u=(u_1,...,u_n) \in (A^*)^n$,
write $u_i=a_{i,1} \cdots a_{i,j_i}$ with each 
$a_{i,m} \in A$ for $1 \le i \le n$ and $1 \le m \le j_i$.
Let $M:=\max\{j_1,...,j_n\}$, and define
$\tilde u_i:=u\$^{M-j_i}$, so that each of 
$\tilde u_1$, ..., $\tilde u_n$ 
has length $M$.  That is, $\tilde u_i$ is
a word over the alphabet $(A \cup \{\$\})^*$, and
we can write $\tilde u_i = c_{i,1} \cdots c_{i,M}$
with each $c_{i,m} \in A \cup \{\$\}$.
The word $\pad(u):=(c_{1,1},...,c_{n,1}) \cdots
(c_{1,M},...,c_{n,M})$ is the {\em padded word}
over the alphabet $A_n$ induced by the $n$-tuple 
$(u_1,...,u_n)$ in $(A^*)^n$.  

A subset $L \subseteq (A^*)^n$ is called
{\em \sr} if the {\em padded extension}
set $\pad(L) := \{\pad(u) \mid u \in L\}$ of
padded words associated to the elements of $L$ is
a regular language over the alphabet $A_n$.
The class of \sr\ languages is closed under
finite unions and intersections, since the
padded extension of a union [resp.~intersection] is
the union [resp.~intersection] of the
padded extensions.
We also include two lemmas on \sr\ languages
for use in later sections.  The first lemma
says that the ``diagonal'' of a regular set
is regular.

\begin{lemma}\label{lem:diagonal}
If $L$ is a regular language over an alphabet $A$,
then the set $\Delta(L):=\{\pad(w,w) \mid w \in L\}$
is a regular language over 
the alphabet $A_2=(A \cup \$)^2 \setminus \{(\$,\$)\}$.
\end{lemma}

\begin{proof}  Given an expression of the regular language
$L$ using letters of $A$ together with the operations
$\cup,\cap,(~)^c,\cdot,(~)^*$, replace every instance
of a letter $a \in A$ with the letter $(a,a) \in A_2$.
\end{proof}

\begin{lemma}\label{lem:product}
If $L_1,...,L_n$ are regular languages over $A$,
then their Cartesian product 
$L_1 \times \cdots \times L_n \subseteq (A^*)^n$
is \sr.
\end{lemma}

\begin{proof}
For each $1 \le i \le n$ define the monoid
homomorphism $\rho_i:A_n^* \ra (A \cup \$)^*$ by
$\rho_i(a_1,...,a_n):=a_i$.  Then the padded
extension of the product language 
$L:=L_1 \times \cdots \times L_n$
satisfies
$\pad(L) = \cap_{i=1}^n \rho_i^{-1}(L_i\$^*)$.
Since each language $L_i\$^*$ is regular, and
regular languages are closed under homomorphic preimage
and finite intersection, then $\pad(L)$ is regular.
\end{proof}

A (deterministic) 
{\em asynchronous (two tape) automaton} over $A$
is a finite state automaton $M=(A \cup \{\#\},Q,q_0,P,\delta)$
satisfying:
(1) The state set $Q$ is a disjoint union
$Q=Q_1 \cup Q_1^\# \cup Q_2 \cup Q_2^\# \cup \{q_f\} \cup \{F\}$
of six subsets, the initial state $q_0$ lies in
$Q_1 \cup Q_2$, and the set of accept states is $P=\{q_f\}$.
(2) The transition function $\delta : Q \times (A \cup \{\#\}) \ra Q$ 
satisfies
$\delta(q,a) \in Q_1 \cup Q_2 \cup \{F\}$ if $q \in Q_1 \cup Q_2$
and $a \in A$;
$\delta(q,a) \in Q_1^\# \cup \{F\}$ if either 
($q \in Q_2$ and $a = \#$) or ($q \in Q_1^\#$ and $a \in A$);
$\delta(q,a) \in Q_2^\# \cup \{F\}$ if either 
($q \in Q_1$ and $a = \#$) or ($q \in Q_2^\#$ and $a \in A$);
$\delta(q,a) \in \{q_f,F\}$ if $q \in Q_1^\# \cup Q_2^\#$
and $a = \#$; and $\delta(q,a)=F$ if 
$q=F$ and $a \in A \cup \{\#\}$.
As before, extend $\delta$ to a function
$\delta:Q \times (A \cup \{\#\})^* \ra Q$ recursively
by $\delta(q,wa):=\delta(\delta(q,w),a)$.

This finite state automaton is viewed as
reading from two tapes rather than one, by the
interpretation that the words on each tape
are to have an ending symbol $\#$ appended, and
when the automaton $M$ is in a state in
$Q_i \cup Q_i^\#$, then $M$ will read the next symbol
from tape $i$.  Then the automaton
 is in a state of $Q_i^\#$ after $M$
 has finished reading the word on
the other tape.

More precisely, given 
a pair of words $(u,v) \in ((A \cup \{\#\})^*)^2$,
a {\em shuffle} of $(u,v)$ is a word 
$u_1v_1 \cdots u_jv_j \in (A \cup \{\#\})^*$ such that 
each $u_i,v_i \in (A \cup \{\#\})^*$,
$u=u_1 \cdots u_j$, and $v=v_1 \cdots v_j$.
Let $(u,v) \in ((A \cup \{\#\})^*)^2$, and
write $u=a_{1,1} \cdots a_{1,m_1}$ and $v=a_{2,1} \cdots a_{2,m_2}$ where
each $a_{i,j}  \in A \cup \{\#\}$.
Given a state $q \in Q$ and the pair $(u,v)$,
there is a unique word 
$\sigma_{M,q}(u,v) := c_1 \cdots c_{m+n}\in (A \cup \{\#,F\})^*$
defined recursively, 
such that 
$c_1:=a_{i,1}$ if $q \in Q_i \cup Q_i^\#$ 
    (and $1 \le m_i$)
and $c_1:=F$ if $q \in \{q_f,F\}$, and
whenever $k \le m+n-1$,
if $c_1 \cdots c_k$ is a shuffle of
$(a_{1,1} \cdots a_{1,k_1},a_{2,1} \cdots a_{2,k_2})$
with $\delta(q,c_1 \cdots c_k)=q'$, then
$c_{k+1}:=a_{i,k_i+1}$ if $q' \in Q_i \cup Q_i^\#$ 
    (and $k_i<m_i$)
and $c_{k+1}:=F$ if $q' \in \{q_f,F\}$; and if $c_k=F$ then
$c_{k+1}:=F$.

A pair $(u,v) \in (A^*)^2$ is accepted by
the \aau\ $M$ if and only if  
$\sigma_{M,q_0}(u\#,v\#)$ is a shuffle of $(u\#,v\#)$;
i.e., there is no occurrence of the letter $F$.
(Equivalently, $(u_1,u_2)$ is in the language of $M$
if and only if the machine $M$ reads
the next letter from the $u_i\#$ tape whenever
$M$ is in a state of $Q_i \cup Q_i^\#$, 
$M$ starts in state $q_0$, and
$M$ ends in state $q_f$ when both tapes have been
read.)
A subset of $A^* \times A^*$ is an {\em \are} 
language if it is the set of word pairs accepted
by an \aau.

Again we include a closure property for
\are\ languages for later use.  This result
is proved by Rabin and Scott in~\cite[Theorem~16]{rabinscott}.

\begin{lemma}~\cite{rabinscott}\label{lem:proj}
If $L \subset (A^*)^2$ is an
\are\ language, then the projection on the first
coordinate given by the set 
$\rho_1(L):=\{u \mid \exists (u,v) \in L\}$
is a regular language over $A$.
\end{lemma}


\section{Autostackable versus automatic:  
Word problems and van Kampen diagrams}\label{sec:aut}


We give a definition of automatic structures 
for groups that is equivalent to, but differs from, the
original definition in~\cite{echlpt}, in order
to illustrate more completely the close connection
to Definition~\ref{def:as} of autostackable 
structures above.  Both automaticity and autostackability
utilize the concepts of a set $\nf$ of normal forms for a group
$G$ over a generating set $A$, but in contrast
to the \wstf\ $\stf$ for autostackability which has
a finite image set, the
definition of automaticity relies on the 
{\em normal form map} 
{nf}$_\nf: \nf \times A \ra A^*$ defined by
{nf}$_\nf(y_g,a) := y_{ga}$.

\begin{definition}\label{def:am}
A group $G$ with finite \sym\  generating
set $A$ is {\em automatic} if
there are a set $\nf$ of normal forms 
for $G$ over $A$ and a constant $k$ such that the following hold:
\begin{itemize} 
\item[(i)]
The graph of the function {\em nf}$_\nf: \nf \times A \ra A^*$, 
$$
graph(\text{\em nf}_\nf):=\{(y_g,a,y_{ga}) \mid g \in G, a \in A\},
$$ 
is a \sr\ language.
\item[(ii)] For each $g \in G$ and $a \in A$,
the pair of paths in $\ga$  
labeled $y_g$ and $y_{ga}$
beginning at the identity vertex $\gi$ 
and ending at the endpoints of $\ega$ must
{\em $k$-fellow travel}; that is, for any natural number $i$,
if $w$ and $w'$ are the length $i$ prefixes of the words
$y_g$ and $y_{ga}$, then there must be a path in $\ga$
of length at most $k$ between the 
vertices of $\ga$ labeled by $w$ and $w'$.
\end{itemize}
\end{definition}

In fact, the definition of automaticity given 
in~\cite[Defn.~2.3.1,Thm.~2.5.1]{echlpt} requires
only property (i) above;
indeed,
in~\cite[Thm.~2.3.5,Thm.~3.3.4]{echlpt}  
Epstein et.~al.~show 
that the geometric
property (ii) follows from the algorithmic property (i). 
Moreover, it is immediate from the properties
of regular languages discussed in Section~\ref{sub:reg} 
that the set $graph($nf$_\nf)$ is a \sr\ language
if and only if
the sets $L_a:=\{(y_g,y_{ga}) \mid g \in G\} \subset (A^*)^2$
are \sr\ for each $a \in A \cup \{\ew\}$,
giving the equivalence of property (i)
above with the definition in~\cite{echlpt}.

Comparing Definitions~\ref{def:as} and~\ref{def:am},
the automatic property (i) requires a finite state automaton that can
recognize the tuple $(y_g,a,z)$ where the third coordinate
is the normal form $z=y_{ga}$, but the
autostackable property (1) requires only a FSA that
recognizes such a tuple in which $z$ is a bounded length
word giving information toward eventually finding the
normal form $y_{ga}$.  (We make this more precise below.)
In analogy with the autostackable property (2)
of Definition~\ref{def:as}, the automatic group 
property (ii) naturally divides into degenerate
and recursive cases, 
in that if the
directed edge $\ega$ is degenerate, we have
the stronger property that the paths $y_g,y_{ga}$
1-fellow travel.

Analogous to the relationship between autostackable
and stackable groups, removing the algorithmic property (i),
a group $G$ is called {\em combable} 
over $A$ if the geometric
property (ii) of Definition~\ref{def:am} holds
for some set $\nf$ of normal forms and some constant $k$.
Note that combability, and hence also automaticity,
imply finite presentability; in particular, the set $R$ of
all words of length up to $2k+2$ that represent
the identity are a set of defining relators for the group.

If $G$ is a combable group satisfying the further 
property that 
the words of the normal form set $\nf$
label simple
paths in the Cayley graph $\ga$, for example in the
case that $\nf$ is closed under taking prefixes, then
the ``seashell'' method discussed in Section~\ref{sub:nfvkd}
extends to the following procedure to construct
a van Kampen diagram
(with respect to the presentation induced by the combable structure)
for any word that represents 
the identity of $G$.
Given a word $w=b_1 \cdots b_n$ representing the
identity of $G$,
with each $b_i \in A$, 
let $y_i:=y_{b_1 \cdots b_i}$ for each $i$.
Property (ii) shows that for each $i$
there is a van Kampen diagram $\Delta_i$ labeled by 
$y_{i-1}b_iy_i^{-1}$
that is ``$k$-thin'' as illustrated in Figure~\ref{fig:ft}.
\begin{figure}
\begin{center}
\includegraphics[width=4.0in,height=0.8in]{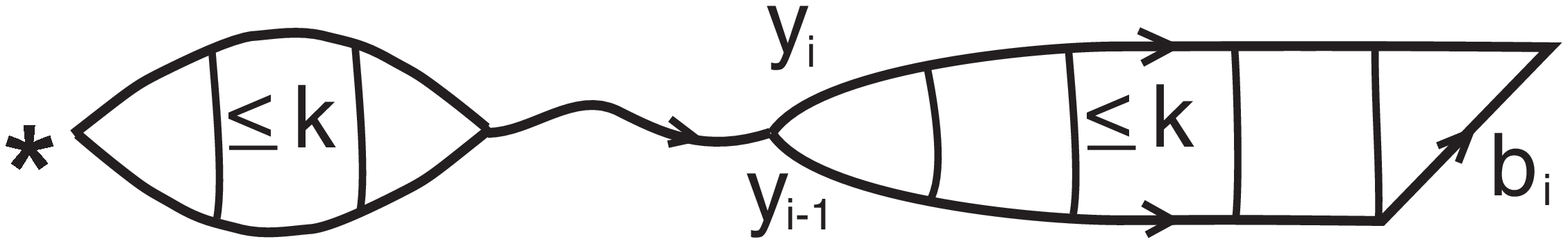}
\caption{``$k$-thin'' van Kampen diagram}\label{fig:ft}
\end{center}
\end{figure}
Gluing these $k$-thin diagrams along
their $y_i$ boundaries results in 
a planar van Kampen diagram for $w$;
see Figure~\ref{fig:seashell}.
In the case that the group is automatic,
this yields a solution of the word problem.
(See~\cite{echlpt} for full details.)

In~\cite{bh}, the first two authors of this paper show that 
every stackable group $G$ also has
a finite presentation and admits
a procedure for building van Kampen
diagrams.  Moreover, 
in the case that the set 
$S_\stf := \{(w,a,\stf(e_{\pi(w),a})) \mid w \in A^*, a \in A\}$
obtained from the \wstf\ $\stf$
is decidable, they show that the procedure is
an effective algorithm and the group has solvable
word problem.  In part (1) of
Definition~\ref{def:as} above,
however, the \sr\  (and hence recursive)
set $graph(\stf)$ is a subset of $S_\stf$, namely
$graph(\stf) =S_\stf \cap (\nf \times A \times A^*)$.
We alter the stacking reduction procedure of~\cite{bh}
to solve the word problem for autostackable groups 
as follows.

For a group $G$ with a stackable structure given
by a set $\nf$ of normal forms over a \sym\ generating
set $A$ and a \wstf\ $\stf:\nf \times A \ra A^*$,
the {\em stacking reduction algorithm} on
words over $A$ is a \prs\ given by
\begin{align*}
R_\stf :=
& \{(ya,y\stf(y,a))  \mid y \in \nf, a \in A, ya \notin \nf \cup A^*a^{-1}a\} \\
& \cup
\{(yaa^{-1},y) \mid ya \in \nf, a \in A\}. 
\end{align*}
Recall that starting from any word $w$ in $A^*$, whenever we can decompose
$w$ as $w=ux$ for some rule $(u,v) \in R_\stf$ and word $x \in A^*$,
then we can rewrite $w \ra vx$.  
Each of these rewritings consists either of free reduction
or {\em $\stf$-reduction}. 

\begin{lemma}\label{lem:asimpliescpr}
If $G$ is a group with \sym\ generating set $A$
and a stackable structure consisting of a normal
form set $\nf$ and a \wstf\ $\stf$, then
the \prs\ $R_\stf$ is a \cpr\ for $G$.
\end{lemma}


\begin{proof}
Let $w$ be any word in $A^*$, and write $w=b_1 \cdots b_m$
with each $b_i$ in $A$. 
Suppose that 
$w'=c_1 \cdots c_n$,
with each $c_j$ in $A$, is obtained from $w$
by repeated applications of 
free and $\stf$-reductions, and that
$w' \ra w''$ is a single instance of another $R_\stf$
rewriting operation.
If the rewriting $w' \ra w''$ is a free reduction,
then two letters of $w'$ are removed, and if
this rewriting is a $\stf$-reduction, then a single
letter of $w'$ is replaced by a bounded length word.
Inductively this shows that each letter of the 
word $w''$
is the result of successive rewritings from a specific
letter $b_i$ of the original word $w$.
Viewing this topologically, if the rewriting operation
$w' \ra w''$ is free reduction, then the directed
path in the Cayley graph $\ga(G,A)$ starting at $\gi$
and labeled $w''$ is obtained from the path labeled 
by $w'$ via the removal of two edges, and if the
rewriting is $\stf$-reduction, 
then a single edge $e'$ of the $w'$ path is replaced
by the path $\ff(e')$, where $\ff$ is the \wff\ induced
by the \wstf\ $\stf$.  In the latter case, there is a specific
recursive edge $e_i:=e_{\pi(b_1 \cdots b_{i-1}),b_i} \in \vec E_{\nf,r}$ 
for some index $1 \le i \le m$
on the path labeled $w$ from $\gi$ in $\ga$
such that $e'$ was obtained from $e_i$ via successive applications
of the \wff, and for each recursive edge $e''$ 
along the path $\ff(e')$, we have $e'' <_\stf e'$,
where $<_\stf$ is the strict well-founded partial ordering
given in Definition~\ref{def:as}(2r).
Since at each application of the \wff\ a bounded number
of recursive edges are added to the path, K\"onig's
Infinity Lemma (see, for example,~\cite[Lemma~8.1.2]{konig})
shows that at most finitely many 
$\ff$-reductions can be applied starting from each of the
finitely many edges of the original path labeled $w$.
Hence only finitely many $\stf$-reductions can be applied
in any sequence of rewritings starting from the word $w$.
Between these $\stf$-reductions, only finitely many
free reductions can occur.
Hence after finitely many
$R_\stf$ rewriting operations,
we must obtain an irreducible word $y_w$,  
and so the \prs\ $R_\stf$ is terminating.

Now suppose that $y$ is any irreducible word with
respect to $R_\stf$.  
Write $y=a_1 \cdots a_n$ with each $a_i$ in $A$
and $y_i:=a_1 \cdots a_i$ for each $i$, and
suppose that 
$y_j$ is the shortest prefix of $y$ that does not lie in $\nf$.
Since the empty word $\ew$ lies
in the normal form set $\nf$ of the stackable
structure, we have $j \ge 1$.
Now $y_{j-1} \in \nf$, and
either $y_{j-1}=y_{j-2}a_j^{-1}$, in which
case a free reduction rule of $R_\stf$ applies to $y$, or
else $y_{j-1}$ does not end with the letter
$a_j^{-1}$,  in which case a $\stf$-reduction rule
applies to $y$. However, this contradicts the
irreducibility of $y$.  Therefore every prefix of the
word $y$, including the word $y$ itself, must lie
in $\nf$.
Thus the set of irreducible words
with respect to $R_\stf$ is contained in
the set $\nf$ of normal forms for $G$.

Next suppose that $w$ is any word in the normal
form set $\nf$.
By the termination proof above, there is a
finite sequence of rewritings from $w$ to
an irreducible word $y_w$.  Since every pair 
of words in the \prs\ $R_\stf$ 
represents the same element of the group $G$, then
$w =_G y_w$.  By the previous paragraph,
the irreducible word $y_w$ must lie in $\nf$.
But since each element of $G$ has exactly one
representative in $\nf$, this implies that $w=y_w$.
Hence the set $\nf$ of normal forms for the 
stackable structure is equal to the set
of irreducible words with respect to the
\prs\ $R_\stf$.
Note that this shows both that the set $\nf$
is prefix-closed, and that the set of $R_\stf$-irreducible
words are a set of normal forms.  Hence
$R_\stf$ is convergent.

Finally, since $A$  is a monoid generating set
for $G$, and
the rules of $R_\stf$ define relations of $G$
that give a set of normal forms for $G$,  
the \cpr\ $R_\stf$ gives a monoid presentation of $G$.
\end{proof}


Recall that the normal form of the identity element
in an autostackable group must be the empty word.
Decidability of the set $graph(\stf)$
implies that
for any word $w \in A^*$, 
one can determine whether or not a
$\stf$-reduction applies, and so
Lemma~\ref{lem:asimpliescpr} completes the word problem solution
in that case.

An immediate consequence of Lemma~\ref{lem:asimpliescpr}
is that the {\em stacking presentation}
\[
G=\langle A \mid \{\stf(y_g,a)a^{-1} \mid g \in G, a \in A\}\rangle
\]
is a (group) presentation for the stackable group $G$; 
property (2) of the definition of stackable
implies that this presentation is finite.
In~\cite[Proposition~1.12]{bh}, the 
first two authors of this paper show how to
use computability of the set $S_\stf$ to
obtain an algorithm for constructing van Kampen
diagrams over this presentation; a similar alteration of 
the proof shows that
this algorithm applies in the case that
$graph(\stf)$ is recursive.
However, in~\cite[Proposition~1.12]{bh}, another hypothesis was
included, that the generating set $A$ of the stackable
structure did not include a letter representing
the identity element of the group.  We note that
given any autostackable 
structure for a
group $G$, with inverse-closed generating set
$A$, normal forms $\nf$, and stacking function $\stf$,
if $A' \subset A$ is the set of letters in $A$
representing $\gi$,
then since the normal form set is prefix-closed,
no element of $\nf$ can contain a letter from $A'$.
It can then be shown that $G$ is also autostackable
over the inverse-closed generating set $B:=A \setminus A'$,
with the same normal form set $\nf$, and 
the stacking function $\stf':\nf \times B \ra B^*$
given by setting $\stf'(y,b)$ equal to the
word $\stf(y,b)$ with all instances of letters
in $A'$ removed.  

Hence we have the following.

\begin{proposition}\label{prop:solvwp}
Autostackable groups are finitely presented,
have solvable word problem, and admit
a recursive algorithm which, upon input of
a word $w \in A^*$ with $\pi(w) = \gi$,
builds a van Kampen diagram for
$w$ over the stacking presentation.
\end{proposition}

We include a few more details here to illustrate
the difference between the van Kampen diagrams
built from an autostackable structure 
and those built from a prefix-closed automatic
structure.
For an autostackable group,
since the set of normal forms is prefix-closed,
each normal form word must label a 
simple path in the Cayley graph $\ga$,
and as in the case of automatic groups,
we extend the ``seashell'' method described
in Section~\ref{sub:nfvkd} to a diagram-building
algorithm.  Given a word $w=b_1 \cdots b_n$
with each $b_i \in A$ and such that 
$\pi(w)=\gi$, 
and letting $y_i:=y_{b_1 \cdots b_i}$ for each $i$,
this method requires an algorithm for
building van Kampen diagrams $\Delta_i$ for 
the words $y_{i-1}b_iy_i^{-1}$, which then
can be glued as in Figure~\ref{fig:seashell}
to obtain the diagram for $w$. 
However, in this case the van Kampen 
diagram  $\Delta_i$ will not be
``thin'' in general, but instead is built
by recursion using 
property (2) of Definition~\ref{def:as}.
If the directed edge $e_{y_{i-1},b_i}$
of $\ga$ is degenerate, then the van Kampen
diagram $\Delta_i$ is homeomorphic to a
line segment, containing no 2-cells;
this is pictured in Figure~\ref{fig:asdeg}.
\begin{figure}
\begin{center}  
$~$\hspace{-2.7in}~\includegraphics[width=1.9in,height=0.5in]{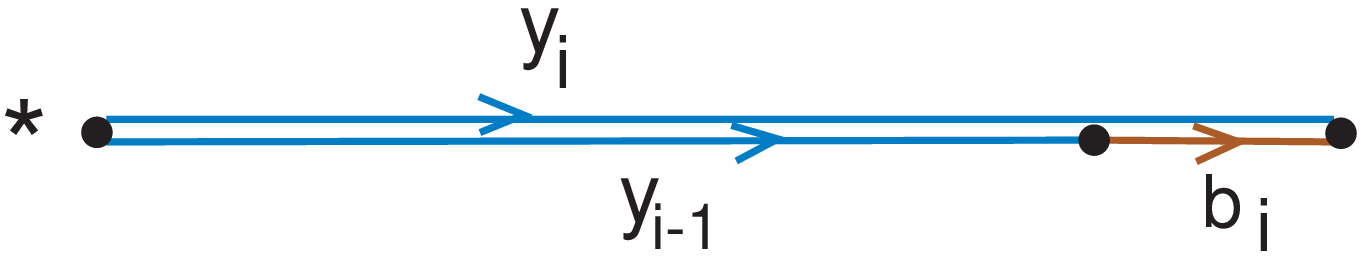}~$~~~~~~~\ \ \ \ \ \ \ \ \ \ \ \ \ \ \ \ \ ~~~~~~~~~~~~~~~~~~$~\includegraphics[width=1.9in,height=0.5in]{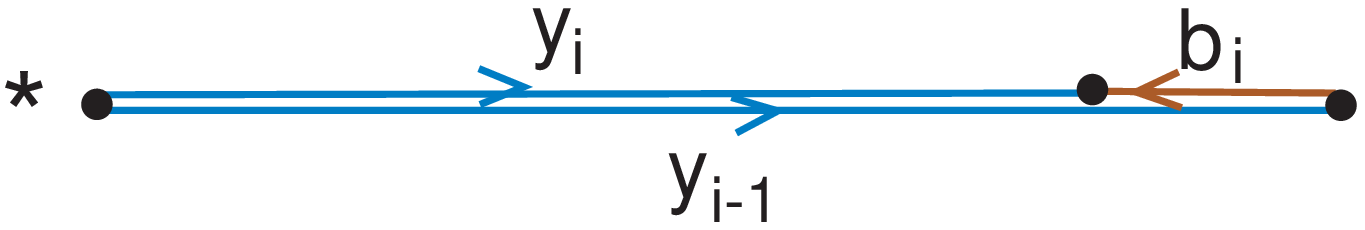}
\caption{Degenerate van Kampen diagrams}\label{fig:asdeg}
\end{center}
\end{figure}
On the other hand, if the edge $e_{y_{i-1},b_i}$
is recursive, and we write
$\stf(y_{i-1},b_i)=a_1 \cdots a_m$ with each
$a_j \in A$, then by Noetherian induction 
(using the well-founded strict partial ordering $<_\stf$)
we may assume
that for each $1 \le j \le m$
we have already built a van Kampen diagram $\Delta_j'$
for the word 
$y_{j-1}'a_jy_j'^{-1}$, 
where $y_j'$ denotes the normal form
word representing the element $y_{i-1}a_1 \cdots a_j$
for each $j$.
Successively gluing these diagrams $\Delta_j'$, 
or {\em stacking} them,
along their common (simple) boundary paths 
$y_j'$,
we obtain a planar diagram with boundary word 
$y_{i-1}\stf(y_{i-1},b_i)y_i^{-1}$.
Finally, glue on a single 2-cell whose
boundary is labeled by the word 
$\stf(y_{i-1},b_i)^{-1}b_i$ to obtain 
the required van Kampen diagram $\Delta_i$.  
This 
process is illustrated in Figure~\ref{fig:asrec}.
\begin{figure}
\begin{center}  
\includegraphics[width=3.4in,height=1.3in]{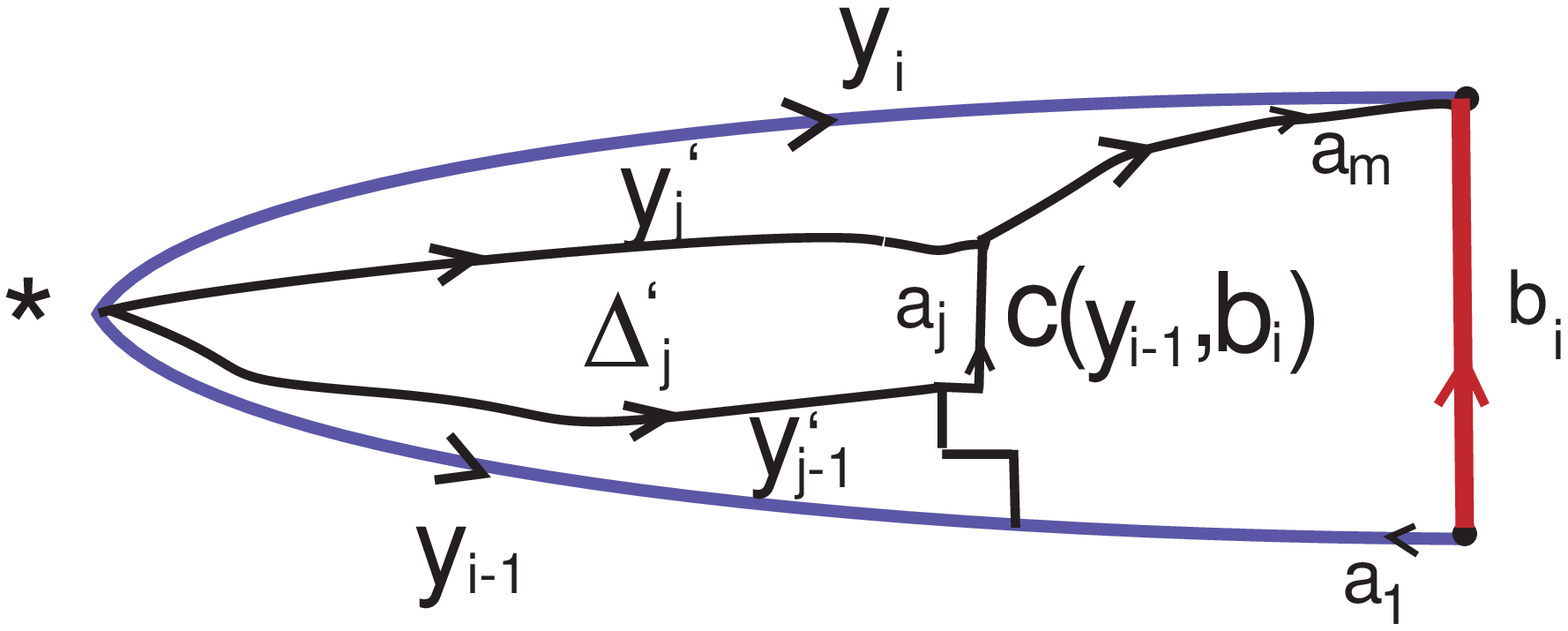}
\caption{Recursive van Kampen diagram}\label{fig:asrec}
\end{center}
\end{figure}


\section{Asynchronously automatic groups} \label{sec:asa}



A group $G$ with finite \sym\ generating set $A$
is \asa\ if there is a regular language $\nf=\{y_g \mid g \in G\}$ of
normal forms for $G$ over $A$ such that for each
$a \in A$ the subset
\[
L_a:=\{(y_g,y_{ga}) \mid g \in G\}
\]
of $A^* \times A^*$ is an \are\ language.
Every automatic group is also \asa.
This section is devoted to the proof of 
Theorem~\ref{thm:asa}.

\begin{theorem}\label{thm:asa} 
Every group that has an \asa\ 
structure with a prefix-closed normal form set
is autostackable.
\end{theorem}

\begin{proof}
Let $G$ be an \asa\ group with 
finite \sym\ generating set $A$ and prefix-closed
normal form set $\nf$, and for each
$a \in A$ let $M_a=(A \cup \{\#\},Q_a,q_{0},P_a,\delta_a)$ 
be an \aau\ accepting the
language $L_a$.
By~\cite[Theorem~7.2.4]{echlpt}, we may also assume
that the \asa\ structure is {\em bounded}.  That is,
there is a constant $C$ such
that for each pair $(u,v) \in L_a$, the 
shuffle $\sigma_{M_a,q_0}(u\#,v\#)$ (see Section~\ref{sub:reg}
for the definition of this notation) of the pair $(u\#,v\#)$
has the form $\sigma_{M_a,q_0}(u,v)=u_1v_1 \cdots u_mv_m$
where $u=u_1 \cdots u_m$, $v=v_1 \cdots v_m$,
each $u_i,v_i \in (A \cup \{\#\})^*$, and the lengths of
these subwords satisfy $0 \le l(u_1) \le C$,
$1 \le l(u_i) \le C$ for all $2 \le i \le m$,
$1 \le l(v_i) \le C$ for all $1 \le i \le m-1$,
and $0 \le l(v_m) \le C$.

By increasing this constant if necessary, we may
also assume that $C$ is greater than 3 and greater than the 
cardinality $|Q_a|$ of the set of states of the
automaton $M_a$ for every $a \in A$.
We can view $M_a$ as a finite graph with vertex set
$Q_a$ and a directed edge labeled $b \in A \cup \{\#\}$ from
$\hat q$ to $\tilde q$ whenever $\delta_a(\hat q,b)=\tilde q$.
Let $\qgood$ be the set of all states $q$
in $Q_{a,1} \cup Q_{a,2}$ such that there
is a path in $M_a$ from the initial state $q_0$ to $q$ and there also is a
path in $M_a$ from $q$ to the accept state $q_f$. For each $q \in \qgood$,
by eliminating repetition of vertices along the path to $q_f$, there
must also be a directed edge path in $M_a$ from $q$ to $q_f$ of
length less than $C$.  Let $W_q \in (A \cup \{\#\})^*$
be a fixed choice of such a word, for each such $q$.
Note that this word must contain two instances of the
letter $\#$, and we can write 
$W_q=\sigma_{M_a,q}(p_q\#,r_q\#)$
for two words $p_q,r_q \in A^*$ satisfying 
$l(p_q)+l(r_q) \le C-2$.

To define the function $\stf:\nf \times A \ra A^*$,
we first set $\stf(y_g,a):=a$ whenever 
the edge $\ega$ lies in the
set $\dgd=\vec E_{\nf,d}$ of degenerate edges 
of the Cayley graph $\ga(G,A)$ with respect to the
set $\nf$ of normal forms; 
i.e.~whenever either
$y_ga=y_{ga}$ or $y_{ga}a^{-1}=y_g$, as required
for property (2d) of Definition~\ref{def:as}.
Now suppose that $\ega$ is recursive.
If $l(y_g)+l(y_{ga}) \le C^2+3C$, then
define $\stf(y_g,a):=y_g^{-1}y_{ga}$.

On the other hand, suppose that 
$l(y_g)+l(y_{ga}) > C^2+3C$.
The pair $(y_g,y_{ga})$ is accepted by the
\aau\ $M_a$, and so
the word $w:=\sigma_{M_a,q_0}(y_g\#,y_{ga}\#)$
is a shuffle of $(y_g\#,y_{ga}\#)$ and
satisfies $\delta_a(q_0,w)=q_{f}$, the accept
state of $M_a$.
The bounded property above implies that
$l(y_{ga}\#) \le Cl(y_g\#)$, and so
$l(y_g\#)+Cl(y_g\#) > C^2+3C+2$, which 
gives $l(y_g) > C+1$.
Write $w=w'w''$ where $w''$ is the shortest suffix of
$w$ containing exactly $C+1$ letters from $y_g$
(i.e., $C+2$ letters from $y_g\#$).
Applying the bounded property again shows that 
$w'=\sigma_{M_a,q_0}(u,v)$ and $w''=\sigma_{M_a,q}(s\#,t\#)$
where $q=\delta_a(q_0,w') \in \qgood$, 
$u,v,s,t \in A^*$, $us=y_g$, $vt=y_{ga}$,
$l(s)=C+1$, and
$1 \le l(t\#) \le C^2+2C$.
Returning to the view of $M_a$ as a finite graph, 
the word $w'$ labels a path from $q_0$ to $q$ and
$w''$ labels a path from $q$ to $q_f$.
Thus the word $\tilde w:=w'W_q$ also labels
a path from $q_0$ to $q_f$, and so the
pair $(up_q,vr_q)$ lies in the language $L_a$.
Note that $usa =_G vt$ and $up_qa =_G vr_q$, and so
$s^{-1}p_qar_q^{-1}t =_G a$.  
Moreover, the pair $(s,t)$ and the state $q$
are uniquely determined
by $(y_g,y_{ga})$; i.e., by $g$ and $a$.
In this case
we define $\stf(y_g,a):=s^{-1}p_qar_q^{-1}t$.

With this definition of the \wstf\ $\stf$,
the length of the word $\stf(y_g,a)$ is at most
$C^2+3C$ for all $g \in G$ and $a \in A$,
and in each case $\stf(y_g,a) =_G a$.

Now suppose that $e=\ega$ and $e'=e_{g',a'}$ are recursive
edges such that $e'$ lies on the directed
path in the  Cayley graph $\ga$ starting
at the vertex $g$ and labeled by the word
$\stf(y_g,a)$.  Now when $l(y_g)+l(y_{ga})  \le C^2+3C$
the path $\stf(y_g,a)=y_g^{-1}y_{ga}$ follows
only degenerate edges, so we must have
 $l(y_g)+l(y_{ga}) > C^2+3C$.  In this case, the
path starting at $g$ and labeled  by the word 
$\stf(y_g,a):=s^{-1}p_qar_q^{-1}t$ defined above
follows only degenerate edges along the subpaths
labeled by
$s^{-1}p_q$ and $r_q^{-1}t$, since 
$us=y_g$, $vt=y_{ga}$, and
$(up_q,vr_q) \in L_a$, and so
the words
$vr_q,vt \in \nf$ as well.
So the edge $e'$ must be the edge
labeled $a'=a$ with initial vertex 
$g'=_G gs^{-1}p_q =_G up_q$.  That is,
we have normal forms $y_g=us$ with $l(s)=C+1$
and $y_{g'}=up_q$ with $l(p_q) \le C$, and 
so the normal form to the initial vertex of $e'$
is strictly shorter than the normal form
to the initial vertex of $e$.
Hence the relation $<_\stf$ defined in
property (2r) of Definition~\ref{def:as}
strictly increases the length of the normal form
of the initial vertex of the edges,
and so is a well-founded strict partial ordering.
Therefore $G$ is stackable over $A$.

By hypothesis the normal form set $\nf$ is a
regular language, and so for each $a \in A$, 
an application of Lemma~\ref{lem:peel}
shows that the language
$J_a:= \{y \mid ya \in \nf\}$ is regular.
Lemma~\ref{lem:product} then shows that 
the languages $J_a \times \{a\} \times \{a\}$
and $(\nf \cap A^*a^{-1}) \times \{a\} \times \{a\}$
are \sr.  The finite union of these sets
for $a \in A$ is the subset
of $graph(\stf)$
corresponding to the application of $\stf$ to
degenerate edges, and therefore this set is \sr.

The subset
\[
L_{smallrec}:=
\{(y_g,a,\stf(y_g,a)) \mid g \in G, a \in A, \ega \in \ves \text{ and }
   l(y_g)+l(y_{ga})  \le C^2+3C\}
\]
of $graph(\stf)$ is finite, and therefore also is \sr.
For use in avoiding overlapping sets later, denote
 $J_{a,smallrec}:=\{y_g \mid (y_g,a,\stf(y_g,a)) \in L_{smallrec}\}$.

For each $a \in A$ and $q \in \qgood$, let
\[
K_{a,q} := \{(u,v) \mid u,v \in A^*, \delta_a(q_0,\sigma_{M_a,q_0}(u,v))=q\},
\]
and note that by definition of $\qgood$ the set $K_{a,q}$ is nonempty.
This subset of $A^* \times A^*$ is \are; in particular, if $q \in Q_{a,1}$, then
$K_{a,q}$ is the accepted language of the \aau\ 
$\widetilde M=(A \cup \{\#\},\widetilde Q,q_0,P_a,\tilde \delta)$
where 
$\widetilde Q_{1}=Q_{a,1}$, $\widetilde Q_{2}=Q_{a,2}$,
$\widetilde Q_{1}^\# = \emptyset$, $\widetilde Q_{2}^\#=\{\tilde q\}$,
and $\tilde \delta(q',b)=\delta(q',b)$ for all
$q' \in \widetilde Q_{a,1} \cup \widetilde Q_{a,2}$ and $b \in A$,
$\tilde \delta(q,\#)=\tilde q$, $\tilde \delta(\tilde q,\#)=q_f$, and
$\tilde \delta(q',b) = F$ otherwise.  The case that $ q \in Q_{a,2}$
is similar.
Then Lemma~\ref{lem:proj} shows that the set
$\rho_1(K_{a,q}) = \{ u \mid \exists (u,v) \in K_{a,q}\}$
is a regular language.

Let 
\[
S_{a,q}:=\{(s,t) \mid s,t \in A^*, l(s)=C+1, \text{ and } 
    \delta(q,\sigma_{M_a,q}(s\#,t\#)) = q_f\}.
\]
The boundedness of the \asa\ structure implies that $l(t)<C^2+2C$, and 
the set $S_{a,q}$ is finite. 
Moreover, note that if $(s,t),(s,t') \in S_{a,q}$, then if we
let $(u,v)$ be an element of the nonempty set $K_{a,q}$,
we have $(us,vt),(us,vt') \in L_a$, and so $vt,vt'$ are
both normal form words representing the same element
$\pi(usa)$ of $G$; hence $t=t'$.  Thus for each $s \in A^{C+1}$,
there is at most one word $t_{q,s}$ such that the pair 
$(s,t_{q,s}) \in S_{a,q}$.

Next for each $a \in A$, $q \in \qgood$, and $(s,t) \in S_{a,q}$,
let
\[
L_{a,q,s} := \rho_1(K_{a,q})s \cap 
[A^* \setminus (J_a \cup (\nf \cap A^*a^{-1}) \cup J_{a,smallrec})].
\]
This is the set of words 
$us \in A^*s$ such that $us \in \nf$, the edge
$e_{\pi(us),a}$ is recursive, $l(us)+l(y_{usa}) > C^2+3C$,
and the path labeled $\sigma_{M_a,q_0}(us,y_{usa})$ goes
from $q_0$ through $q$ to $q_f$ in $M_a$.
Closure properties of regular sets shows that
this language is regular.
Applying Lemma~\ref{lem:product} again,
the language $L_{a,q,s} \times \{a\} \times \{s^{-1}p_qar_q^{-1}t_{q,s}\}$
is a \sr\ subset of $graph(\stf)$ corresponding to these
recursive edges.

We can now write the graph of the \wstf\ $\stf$ as
the finite union
\begin{align*}
graph(\stf) = [\cup_{a \in A} & 
  ((J_a \cup (\nf \cap A^*a^{-1}) \times \{a\} \times \{a\})] \\
& \cup
L_{smallrec} \\
& \cup
[\cup_{a \in A, q \in \qgood, (s,t_{q,s}) \in S_{a,q}}
  (L_{a,q,s} \times \{a\} \times \{s^{-1}p_qar_q^{-1}t_{q,s}\})].
\end{align*}
Closure of the class of \sr\ sets under finite unions then
shows that $graph(\stf)$ is \sr.
Thus $G$ is autostackable.
\end{proof}


\section{Rewriting systems} \label{sec:fcrs}


In this section we prove the characterization of
autostackable groups in terms of \sr\ bounded {\cpr}s,
and conclude with a discussion of {\fcr}s.
We begin by discussing a process for minimizing
{\prs}s.

\begin{definition}\label{def:min}
A \cpr\ $R \subset A^* \times A^*$ for a group $G$
is {\em \pro} if: 
\begin{itemize}
\item[(a)] For each $a \in A$ there is a letter in $A$,
which we denote $a^{-1}$, such that
$\pi(a)^{-1}=\pi(a^{-1})$ (where $\pi:A^* \ra G$ is the canonical map).
\item[(b)] For each pair $(u,v) \in R$, every 
proper prefix of $u$ is irreducible with respect to
the rewriting operations of $R$.
\item[(c)] Whenever $(u,v_1),(u,v_2) \in R$, then $v_1=v_2$.
\end{itemize}
\end{definition}

For any \prs\ $R$ over $A$, let Irr($R$) denote
the set of irreducible words with respect to the
rewriting operations $ux \ra vx$ whenever
$(u,v) \in R$ and $x \in A^*$.  Note that
every prefix of a word in Irr($R$) must also
lie in Irr($R$).

\begin{proposition}\label{prop:min}
If a group $G$ admits a \sr\ \bcpr\ $R$ over a monoid generating set $B$,
then $G$ also admits a \pro\ \sr\ \bcpr\ $Q$ over 
the generating set $A:=B \cup B^{-1}$, such that Irr($R$) = Irr($Q$).
\end{proposition}

\begin{proof}
Let $R$ be a \bcpr\ over $B$ for $G$, and let
$\pi:B^* \ra G$ be the associated surjective monoid homomorphism. 
For each
element $b \in B$, let the symbol $b^{-1}$ denote
another letter, and let $A:=B \cup \{b^{-1} \mid b \in B\}$.
%
%
For each $b \in B$, let $z_b$ denote the unique word in
Irr($R$) representing the element $\pi(b^{-1})$ of $G$.

Let 
\[
R' :=
\{(yb,v)  \mid (yb,v) \in R, y \in \text{Irr}(R), b \in B
\};
\]
i.e., the set of all rules of $R$ whose left entry has every
proper prefix irreducible; i.e., that satisfies property (b)
of Definition~\ref{def:min}.

Let $k$ be the constant associated to the bounded
property of the \prs\ $R$.  Then by expressing the finite set 
\[
W := 
\{(s,t) \in B^{\le k} \times B^{\le k} \mid
\text{the first letters of }s \text{ and }t \text{ are distinct}\},
\]
as $W=\{(s_1,t_1),...,(s_n,t_n)\}$, we can write
each element $r:=(u,v)$ of $R$ 
in the form
$r=(ws_{i(r)},wt_{i(r)})$ for a unique index $i(r) \in \{1,...,n\}$
and word $w \in B^*$. 
For each $1 \le i \le n$, let 
\[R_i' := \{r \in R' \mid i(r)=i
\}.
\]
Then $R'$ is the disjoint union $R' = \cup_{i=1}^n R_{i}'$.
Note that if there are two pairs $r_1=(u,v_1),r_2=(u,v_2) \in R'$
that have the same left hand entry but $v_1 \neq v_2$ on the right, 
then the indices $i(r_1) \neq i(r_2)$ must also be distinct.
Let
\[
Q' :=
\{r=(u,v) \in R' \mid \not\exists \tilde r = (u,\tilde v) \in R' \text{ with } 
      i(\tilde r) < i(r)\}.
\]
That is, the subset $Q'$ of $R$ satisfies properties
(b) and (c) of Definition~\ref{def:min}. 
We then define the \prs\  $Q:=Q' \cup Q''$ where
\[
Q'' :=
\{(yb^{-1},yz_b) \mid y \in \text{Irr}(R) \setminus B^*b\} 
\cup
\{(ybb^{-1},y) \mid yb \in \text{Irr}(R), b \in B\}. 
\]

Suppose that $w \in A^*$ is rewritten 
by a sequence of applications of rewriting operations
using the \prs\ $Q$.  Since 
the only occurrences of letters of $B^{-1}$ in $Q$
appear in left hand sides of pairs in $Q''$,
at most $l(w)$ of the rewritings in this sequence 
involve a rule of $Q''$.
The rules in $Q'$ all lie in the \cpr\ $R$,
which satisfies the termination property, and
so only finite sequences of applications of $Q'$ rules
can occur.  Hence there can be at most
finitely many rewritings in any such rewriting of $w$; that is,
the \prs\ $Q$ is terminating.

Suppose that $w$ is any word in Irr($R$).  Then $w \in B^*$,
so $w$ can't be rewritten using a pair from $Q''$,
and since $Q' \subseteq R$, the word $w$ also can't be 
reduced using $Q'$.  Hence Irr($R$) $\subseteq$ Irr($Q$).

On the other hand, suppose that $x$ is a word in 
Irr($Q$) $\setminus$ Irr($R$).
If $x \in B^*$, then write  $x=x'bx''$
where $x' \in B^*$, $b \in B$, and $x'b$ is the
shortest prefix of $x$ that does not lie in Irr($R$).
But then there must be a pair $(u,v) \in R$ and a word
$z \in B^*$ such that $x'b=uz$.  Since $x'$ is
irreducible over $R$, then $x'b=u$ and $z=\ew$, and the rule
$(u,v)$ also lies in the subset $R'$ of $R$.
Hence there also is a rule $(u,v')$ in the subset $Q'$
of $Q$, since the sets of left hand sides of rules
of $R'$ and $Q'$ are the same.
But then $x$ is reducible over $Q$.
This contradiction implies that 
$B^* \cap$ Irr$(Q) \subseteq$ Irr($R$), and so
we must have
 $x \notin B^*$.  In this case we can write
$x = x'ax''$ with $a \in B^{-1}$ and $x' \in B^*$,
where $a$ is the first occurrence of a letter of $A \setminus B$
in $x$.
Since the set of irreducible words over a \prs\ is
prefix-closed, the word $x'$ lies in Irr($Q$) $\cap B^*$,
and hence also in Irr($R$).  But then the word $x$ can
be reduced using an element of $Q''$, another 
contradiction.  Therefore we have Irr($R$) = Irr($Q$).  
Since the set Irr($R$) is a set of normal
forms for the group $G$, and whenever $(u,v) \in Q$ we
have $u=_Gv$,
this shows that $Q$ is a \cpr\  for the group $G$.

Since the \cpr\ $R$ is bounded with constant $k$,
the rules of the \prs\ $Q$ are also bounded, with 
constant given by the maximum of $k$, $2$, and 
$\max\{l(z_b) \mid b \in B\}$.  

Note that the \prs\ $Q$ has been chosen
to satisfy properties 
(a), (b), and (c) of
the Definition~\ref{def:min}, and so $Q$ is a \pro\ \bcpr.

If moreover the set $R$ is also \sr, then
the padded extension set $\pad(R)=\{\pad(u,v) \mid (u,v) \in R\}$
is a regular language over the alphabet
$B_2=(B \cup \$)^2 \setminus \{(\$,\$)\}$.
Define the monoid homomorphism
$\rho_1:B_2^* \ra B^*$ by $\rho((b_1,b_2)):=b_1$
if $b_1 \in B$ and $\rho((b_1,b_2)):=\ew$ if
$b_1=\$$.  The set Irr($R$) is the language
Irr($R$) = $A^* \setminus (\rho_1(\pad(R))A^*)$;
using closure of regular languages under
homomorphic image, concatenation, and complement
(see Section~\ref{sub:reg} for more on
regular languages),
then Irr($R$) is a regular set over the alphabet $B$.
But then Irr($R$) is also regular over any alphabet
containing $B$, including $A$.

For each $b \in B$,
Lemma~\ref{lem:peel} says that the set
$L_b:=\{y \mid yb \in $Irr$(R)\}$ also is
regular.
Also recall from Lemma~\ref{lem:diagonal} that whenever
$L$ is a regular language over $B$, then the diagonal
set $\Delta(L):=\{\pad(y,y) \mid y \in L\}$ is a 
regular language over $B_2$.
Now the padded extension of the subset $Q''$ of
the \prs\ $Q$ has the decomposition
\[
\pad(Q'') =
\cup_{b \in B} [
(\Delta(\text{Irr}(R) \setminus B^*b) \cdot \pad(b^{-1},z_b))
\cup
(\Delta(L_b) \cdot \pad(bb^{-1},\ew))   ].
\]
Again applying closure properties (in particular
under finite unions) of regular
languages, this shows that $Q''$ is \sr.

Analyzing the subset $Q'$ of $Q$ requires a
few more steps.  First we note that the padded extension
of the set of rules in $R$ satisfying property (b)
in Definition~\ref{def:min} is
$\pad(R') = \pad(R) \cap \rho_1^{-1}($Irr$(R) \cdot B)$,
and so $\pad(R')$ is a regular set.

Next for each $1 \le i \le n$ (where $n=|W|$),
let $L_i:=\rho_1(\pad(R') \cap (\Delta(B^*)\cdot \pad(s_i,t_i)))$
be the set of left hand entries of all of the
rules $r$ in $R_i'$.  
Again closure properties show that $L_i$ is a regular
language over $B$.  
Then the set 
\[
L_i' := L_i \setminus (\cup_{j=1}^{i-1} L_j)
\]
is the set of all left hand entries of elements
$q$ in $Q'$ such that the index $i(q)=i$.
Now Lemma~\ref{lem:peel} shows that the set
$L_i'':=\{y \mid ys_i \in L_i'\}$ is regular.
Putting all of these together, the padded
extension of the set $Q'$ has the decomposition
\[
\pad(Q')=\cup_{i=1}^n \Delta(L_i'') \cdot \pad(s_i,t_i).
\]
Thus $\pad(Q')$ is a regular language over the
alphabet $B_2$, and hence also over the set
$A_2$.  
Hence $Q'$ also is \sr.

Finally the closure of \sr\ sets under finite
unions shows that the \bcpr\ $Q$ is \sr, as required.
\end{proof}

Note that whenever $R$ is a
 \pro\ \cpr\  over an alphabet $A$
and $w \in A^*$ is a reducible word, then
there exists exactly one rewriting operation
(of the form $w=ux \ra vx$ for some $(u,v) \in R$)
that can be applied to $w$.  Hence
for each $w \in A^*$, we can define the
{\em prefix-rewriting length} $prl_R(w)$ to
be the number of rewriting operations required
to rewrite $w$ to its normal form via $R$.

\begin{theorem}\label{thm:prs}
Let $G$ be a finitely generated group. \\
(1) The group $G$ 
is stackable if and only if $G$ admits a bounded
\cpr.\\
(2) The group $G$ 
is autostackable if and only if $G$ admits a \sr\ bounded
\cpr.
\end{theorem}

\begin{proof}
Suppose first that the group $G$ is stackable over 
an \sym\ generating set $A$,
with normal form set $\nf$,
constant $k$, and  \wstf\ $\stf:\nf \times A \ra A^*$
such that the length of $\stf(y,a)$ is at most $k$ for all
$(y,a) \in \nf \times A$.
In Lemma~\ref{lem:asimpliescpr}, we show that
\begin{align*}
R_\stf :=
& \{(ya,y\stf(y,a))  \mid y \in \nf, a \in A, ya \notin \nf \cup A^*a^{-1}a\} \\
& \cup
\{(yaa^{-1},y) \mid ya \in \nf, a \in A\}. 
\end{align*}
is a \cpr\ for the group $G$.  
(Moreover, the irreducible words are the normal forms
from the stackable structure; i.e., Irr$(R_\stf)=\nf$.)
The bound $k$
on lengths of words in the image of $\stf$ implies 
that 
$R_\stf$ is a \bcpr.

If moreover $G$ is autostackable, so that the
set $graph(\stf)$ is \sr, let 
$
\pad(graph(\stf)):=\{\pad(y_g,a,\stf(y_g,a))  \mid g \in G, a \in A\}
$
be the regular language of padded words over 
$A_3=(A \cup \$)^3 \setminus \{(\$,\$,\$)\}$
associated to the elements of the set $graph(\stf)$.
Define the monoid homomorphism 
$\rho_1:A_3^* \ra A^*$
by $\rho_1((a_1,a_2,a_3)):=a_1$ if $a_1 \in A$
and $\rho_1((a_1,a_2,a_3)):=\ew$ if $a_1=\$$, 
and the monoid homomorphism
$\rho_{2,3}:A_3^* \ra ((A \cup \$)^2)^*$ by
$\rho_{2,3}((a_1,a_2,a_3)):=(a_2,a_3)$,
for each
$(a_1,a_2,a_3) \in A_3$.
The normal form set Irr$(R_\stf)=\nf=\rho_1(\pad(graph(\stf)))$ is
the image of a regular set, and so 
is regular.  For each $a \in A$, the set 
$J_a:=\{y \in A^* \mid ya \in \nf\}$ is regular,
applying Lemma~\ref{lem:peel}.
Also with this notation, for each $a \in A$ and $u \in A^{\le k}$
we can write the set $L_{a,u}$ of all
normal form words $y \in \nf$ such that the 
\wstf\ $\stf$ maps $(y,a)$ to the word $u$ as
\[
L_{a,u} := \rho_1(\pad(graph(\stf)) \cap \rho_{2,3}^{-1}(\pad(a,u)\cdot (\$,\$)^*)).
\]
Recalling the fact that the class of 
regular languages is closed under finite intersections
and homomorphic image and preimage,
then since the language $\pad(a,u)\cdot (\$,\$)^*$ over
$(A \cup \$)^2$ is regular, the set $L_{a,u}$ is regular.
Using the notation $\Delta(L)=\{\pad(w,w) \mid w \in L\}$
for any language $L$, we can now decompose the padded
extension of the \prs\ as
\[
R_\stf = 
[\cup_{a \in A,u \in A^{\le k},a \neq u}
    \Delta(L_{a,u}) \cdot \pad(a,u)] 
\cup
[\cup_{a \in A} \Delta(J_a) \cdot \pad(aa^{-1},\$)].
\]
From Lemma~\ref{lem:diagonal},
the languages $\Delta(L_{a,u})$ and $\Delta(J_a)$
over $(A \cup \$)^2$
are regular.  
Since singleton sets are regular, 
and the class of regular languages is
also closed under concatenation  and
finite unions, this decomposition shows that the set $\pad(graph(\stf))$ 
is regular.  Therefore $R_\stf$ is a \sr\ \bcpr\ for the
autostackable group $G$.

Conversely, suppose that the group $G$ admits a \bcpr.
From the proof of Proposition~\ref{prop:min}, there exists a 
\pro\ \bcpr\ $R$, over an \sym\ alphabet $A$,
for the group $G$.  Let $k$ be the constant associated
to the bounded property of this \prs.
Let $\nf$ be 
the set Irr($R$) of words that are irreducible with respect to
the rewriting operations 
$ux \ra vx$ whenever $(u,v) \in R$ and $x \in A^*$.
Since the \prs\ is convergent, then $\nf$ is a set of 
normal forms for $G$.
Note that the empty word and
any prefix of an irreducible word are irreducible,
and so $\nf$ is a prefix-closed language of normal forms
for $G$ over $A$ that contains the empty word.

Define the function $\stf:\nf \times A \ra A^*$
as follows.  For each $y \in \nf$ and $a \in A$, define
$\stf(y,a):=a$ if either $ya \in \nf$ or $y \in A^*a^{-1}$,
as required for
property (2d) of Definition~\ref{def:as}.
If neither of these conditions hold, then the word 
$ya$ is reducible.  Since the maximal prefix $y$
is irreducible, any rule of the \prs\ that applies
to the word $ya$ must have the entire word $ya$
as its left entry.
Because this \prs\ is processed, there is exactly one 
element of $R$ of the form $(ya,v)$ for some $v \in A^*$.
Moreover, there are words $s,t \in A^{\le k}$ and $w \in X^*$
such that $ya=wsa$,  $v=wt$, and (by taking $w$ to be
as long as possible) the words $s$ and $t$ do not start with
the same letter.  In this case we define
$\stf(y,a):=s^{-1}t$, where $s^{-1}$ is the formal inverse
of $s$ in $A^*$.
For every $y \in \nf$ and $a \in A$, then, the length 
of the word $\stf(y,a)$ is at most $2k$, and 
since $wsa=_Gwt$ in the rewriting presentation of $G$, 
we have $\stf(y,a)=_Ga$.  

Let $\ga$ be the Cayley graph for the group
$G$ with generating set $A$, and let 
$\ves=\vec E_{\nf,r}$
denote the
set of 
recursive
edges with respect to the normal form set $\nf$.
Given any directed edge $\ega$ of the Cayley
graph $\ga(G,A)$ with $g \in G$ and $a \in A$, 
let $prl_R(\ega):=prl(y_ga)$ denote the
prefix-rewriting length over $R$ of the associated 
word $y_ga$, where $y_g$ is the irreducible normal
form for $g$.

Suppose that $\ega$ is any edge in $\ves$,
and that $e'$ is an edge on the directed
path in $\ga$ labeled by the word $\stf(y_g,a)$
and starting at the vertex $g$.  Then the word
$y_ga$ is not in normal form, and there is a rule
$y_ga=wsa \ra wt$ in the \prs\ $R$ such that 
$\stf(y_g,a)=s^{-1}t$.
Since the word $y_g$ is in normal form,
the prefix $s^{-1}$ of the word 
$\stf(y_g,a)$
labels a path in $\ga$ starting at
the vertex $g$ that follows only
degenerate edges, in the maximal tree
defined by the normal form set $\nf$.
Writing the word $t=b_1 \cdots b_n$
with each $b_i$ in $A$, then
$e'=e_{gs^{-1}b_1 \cdots b_{i-1},b_i}=
   e_{wb_1 \cdots b_{i-1},b_i}$ 
for some $i$.
Now the sequence of rewriting operations
with respect to the \prs\ $R$ of the 
word $y_ga$ has the form
$y_ga =wsa \ra wt=wb_1 \cdots b_n
\ra^* y_{gs^{-1}b_1 \cdots b_{i-1}}b_i\cdots b_n
\ra^* y_{ga}$,
where $\ra^*$ denotes a finite number (possibly 0)
of applications of rewriting rules, since 
no rewriting operation over the \pro\ \prs\ $R$ 
can be applied affecting
the letter $b_i$ in these words until the prefix
to the left of that letter has been rewritten into 
its irreducible normal form.
Hence the number of rewritings needed
to obtain an irreducible word starting from the word 
$y_ga$ is strictly greater than the number 
required to obtain a normal form starting from
the word $y_{gs^{-1}b_1 \cdots b_{i-1}}b_i$.
That is,
$prl_R(e')<prl_R(e)$.  Then the
usual strict well-founded partial ordering
on the natural numbers implies that the
relation $<_\stf$ of property (2r) in 
Definition~\ref{def:as} is a strict well-founded
partial ordering.  Hence property (2)
of the Definition~\ref{def:as} of autostackable holds,
and so the group $G$ is stackable.

If moreover $G$ has a \bcpr\ that is \sr, 
then Proposition~\ref{prop:min} says that 
there is a \pro\ \sr\ \bcpr\ $R$ over a \sym\ 
generating set $A$ for
the group $G$.
Synchronous regularity of $R$ means that
the set 
$\pad(R)=\{\pad(u,v) \mid (u,v) \in R\}$ of padded words
is a regular language over the set 
$A_2=(A \cup \$)^2 \setminus \{(\$,\$)\}$.
Let $\rho_1:A_2^* \ra A^*$
be the monoid homomorphism defined by 
$\rho_1(a_1,a_2):=a_1$ if $a_1 \in A$ and $\rho_1(a_1,a_2):=\ew$
if $a_1=\$$.
The set $\nf$ of irreducible words with respect to $R$
can then be written as
\[
\nf = A^* \setminus (\rho_1(\pad(R)) A^*),
\]
and so $\nf$ is a regular language. 

For each $a \in A$, an application of Lemma~\ref{lem:peel}
shows that the language
$L_a:= \{y \mid ya \in \nf\}$ is regular.
Lemma~\ref{lem:product} then shows that 
the languages $L_a \times \{a\} \times \{a\}$
and $(\nf \cap A^*a^{-1}) \times \{a\} \times \{a\}$
are \sr.  Thus the subset of $graph(\stf)$
corresponding to the application of $\stf$ to
degenerate edges is \sr.

Given $a \in A$, let $W_a$ be the finite set of all
pairs $(s,t)$ such that $s,t \in A^{\le k}$,
$s$ and $t$ begin with different letters of $A$,
and $s$ does not end with the letter $a^{-1}$. 
Let $\Delta(A^*):=\{(w,w) \mid w \in A^*\}$;
by Lemma~\ref{lem:diagonal}, this language over
$A_2$ is regular.  
For each $(s,t) \in W_a$, let 
\[
P_{a,s,t}:=\rho_1(\pad(R) \cap (\Delta(A^*)\cdot \pad(sa,t))),
\]
which is again regular using the closure properties of
regular languages.
Then the set 
of all words
$w$ such that the rule $(wsa,wt)$ lies in $R$ is
\[
L_{a,s,t}:= \{w \mid wsa \in P_{a,s,t}\},
\]
which is also regular (by Lemma~\ref{lem:peel}).
Applying Lemma~\ref{lem:product} once more
shows that the subset $(L_{a,s,t} \cdot s) \times \{a\} \times \{s^{-1}t\}$
of $graph(\stf)$ corresponding to these recursive
edges is also \sr.

We can now write the  graph of the \wstf\ $\stf$ as
\begin{align*}
graph(\stf) = 
\cup_{a \in A} & [(L_a \times \{a\} \times \{a\}) \cup
((\nf \cap A^*a^{-1}) \times \{a\} \times \{a\})] \\
& \cup_{a \in A, (s,t)\in W_a} (L_{a,s,t} \cdot s) \times \{a\} \times \{s^{-1}t\}.
\end{align*}
Closure of the class of \sr\ languages under
finite unions then implies that $graph(\stf)$ is \sr.
Hence property (1) of Definition~\ref{def:as}
of autostackability also holds in this case.
\end{proof}

Rewriting systems that are not ``prefix-sensitive'', allowing 
rewriting rules to be applied anywhere in a word, have 
been considerably more widely studied and applied in the
literature than {\prs}s.
A {\em \fcr} for a group $G$ consists of a finite 
set $A$
together with a finite subset $R \subseteq A^* \times A^*$
such that 
as a monoid, $G$ is presented by 
$G = Mon\langle A \mid u=v$ whenever $u \ra v \in R \rangle,$
and the rewritings
$xuz \ra xvz$ for all $x,z \in A^*$ and $(u,v)$ in $R$ satisfy:
\begin{itemize}
\item {\em Normal forms:} Each $g \in G$ is 
represented by exactly one {\em irreducible} word 
(i.e. word that cannot be rewritten)
over $A$.
\item {\em Termination:}
There does not exist an infinite sequence of
rewritings $x \ra x_1 \ra x_2 \ra \cdots$.
\end{itemize}

The key difference here is that a rewriting
system allows rewritings $xuz \ra xvz$ for all $x,z \in A^*$
and $(u,v) \in R$, but a
\prs\ only allows rewritings $uz \ra vz$ for all $z \in A^*$
and $(u,v) \in R$.  However, every \fcr\ gives rise
to a \bcpr, yielding the following.

\begin{corollary}\label{cor:fcrs} 
Every group that admits a finite complete rewriting system
is autostackable.
\end{corollary}

\begin{proof}
Given a \fcr\ $R$ for a group $G$ over a
generating set $A$, 
the \prs\ over $A$ defined by
\[
\hat R := \{(wu,wv) \mid (u,v) \in R, w \in A^*\}
\]
allows exactly the same rewriting operations as
the original \fcr, and therefore is a \cpr.
Since the set $R$ is finite, this \prs\ $\hat R$
is also bounded.  Finally, the padded extension of
the set $\hat R$
can be written as 
$\pad(\hat R) = \cup_{(u,v) \in R} \Delta(A^*) \cdot \pad(u,v)$,
and so this set is \sr.
Theorem~\ref{thm:prs}(2) now completes the proof.
\end{proof}

\section*{Acknowledgments}

This work was partially supported by 
a grant from the Simons Foundation (Grant
Number 245625 to the second author).




\end{document}